\documentclass[11pt]{article}

\usepackage{amsmath}
\usepackage{amsthm}
\usepackage{amssymb}
\usepackage[numbers,sort&compress]{natbib}
\usepackage{mathtools}
\usepackage{hyperref}
\usepackage{xcolor}

\theoremstyle{plain}
\newtheorem{theorem}{Theorem}
\newtheorem{lemma}[theorem]{Lemma}
\newtheorem{corollary}[theorem]{Corollary}

\theoremstyle{definition}

\theoremstyle{remark}

\numberwithin{equation}{section}

\newcommand{\z}{\zeta}

\newcommand{\C}{\mathbb{C}}
\newcommand{\Z}{\mathbb{Z}}
\newcommand{\HH}{\mathcal{H}}
\newcommand{\calS}{\mathcal{S}}

\newcommand{\bra}[1]{\left(#1\right)}                
\newcommand{\sqbra}[1]{\left[#1\right]}              

\title{\bf A symmetric extension of Zhao's series for $\zeta(5)$}
\author{Noam Shalev}

\date{}
  
\newcommand{\Address}{{
		\bigskip
		\footnotesize
		N.~Shalev, \textsc{Einstein Institute of Mathematics,
			The Hebrew University of Jerusalem, Israel}\par\nopagebreak
		\textit{E-mail address}: \texttt{noam.shalev2@mail.huji.ac.il}
		
}}

\begin{document}
\maketitle

\begin{abstract}
	We prove a four-parameter  extension of Zhao's series for $\z(5)$. Restricting a Wilf–Zeilberger seed of Au to a hyperplane, we obtain a summand in five parameters with zero sum that is symmetric under all their permutations. Summing it from two different starting points gives two trigonometric closed forms, which we prove by Guillera's periodicity argument. The proof is self-contained, and also gives a companion series for $1/\pi^2$, which we have not found in the literature. Specializing the parameters at roots
	of unity, we obtain Borwein--Bradley-type identities whose products
	depend only on $x^2$, $x^3$, $x^4$ or $x^5$. Their Taylor coefficients are Ap\'ery-like series for odd zeta values and for products of zeta values. In particular, we prove five identities conjectured by Au.
\end{abstract}
\noindent\textbf{Keywords:} Zhao's series, $\zeta(5)$, Ramanujan-type
series, WZ method, Borwein--Bradley identities, harmonic numbers

\section{Introduction}
Ramanujan's series for $1/\pi$ \cite{ramanujan1914} have analogues for
higher powers of $\pi$. Guillera found and proved the first series for
$1/\pi^2$ \cite{guillera2002wz,guillera2003new}, and series for $1/\pi^4$
were found by J.~Cullen and by Y.~Zhao (see the survey \cite{cohenguillera2021}).
Guillera and Rogers \cite{guillerarogers2014} observed that inverting the
Pochhammer quotient of such a series, which gives its \emph{upside-down}
counterpart, can produce a convergent series for a zeta value, even when
the original series diverges. For instance, the Amdeberhan--Zeilberger
series for $\zeta(3)$ \cite{amdeberhanzeilberger1997} is the upside-down
counterpart of a divergent series for $16/\pi^2$
\cite[(35)]{guillera2016bilateralsums}. In the same way, Zhao found the identity
\cite{zhao2017mo}
\begin{equation}\label{eq:zhao}
\sum_{n \geq 1} \frac{a_n \, p(n)}{n^9} = -380928 \z(5), 
\end{equation}
where
\begin{align*}
a_n & = \bra{- \frac{1024}{3125}}^n \frac{(1)_n^9}{(\frac12)_n^5 (\frac15)_n (\frac25)_n (\frac35)_n (\frac45)_n},\\
p(n) & = 5532n^4 - 5600n^3 + 2275 n^2 - 425n+ 30.
\end{align*}
Here $(z)_n = z(z + 1) \cdots (z + n - 1)$, with $(z)_0 = 1$, denotes the rising factorial.
It is the upside-down counterpart of a divergent
Ramanujan-like series for $1280/\pi^4$
(see \cite{guillera2016bilateralsums}). Cohen and Guillera
\cite{cohenguillera2021} list it next to the analogous series\footnote{Curiously,
its polynomial $c(n)\coloneqq 483n^2-245n+30$ reappears in the coefficient of $e_2^2$
of the polynomial $P$ (Section~\ref{sec:extension}) and in the quadratic
identity \eqref{eq:odd-generating}.}
\begin{equation}\label{eq:zeta3}
\sum_{n \ge1}\left(-\frac{256}{3125}\right)^n
\frac{(1)_n^5}{(\tfrac12)_n(\tfrac15)_n(\tfrac25)_n(\tfrac35)_n(\tfrac45)_n}
\frac{483 n^2 - 245 n + 30}{n^5} = -896\,\zeta(3),
\end{equation}
conjectured by Guillera and proved by Au \cite[Example IX]{au2025wzseeds}.
It is remarked in \cite{au2025wzseeds} that \eqref{eq:zhao} appears to
be the only known hypergeometric series for $\zeta(5)$ whose terms
decrease geometrically, and Broadhurst used it for record computations
of $\zeta(5)$ \cite{broadhurst2024five}. A proof of \eqref{eq:zhao} by
a Wilf--Zeilberger seed is given in \cite[Example X]{au2025wzseeds}; it
combines the WZ pair with two further nontrivial evaluations, of a
series equal to $2(93 \zeta(5) - 7 \pi^2 \zeta(3))$ and of a series for
$14 \zeta(3)$. In \cite[Example X]{au2026harmonic}, the seed gives a five-parameter family of series, whose Taylor coefficients lead to conjectural harmonic-number identities extending the same series.

In this short note, we study a four-parameter restriction of that family. In suitable coordinates, its summand is symmetric in five parameters whose
sum is zero, and summing it from two different starting points gives
two trigonometric closed forms (Theorem~\ref{thm:main}).
The first leads to \eqref{eq:zhao} (Corollary~\ref{cor:main}). The second
gives a companion series for $1/\pi^2$ (Corollary~\ref{cor:companion}), which can be written as
\begin{equation}\label{eq:companion-intro}
\sum_{n\ge0}(-1)^n\frac{(2n)!^{10}\,(5n)!}{n!^{15}\,(10n)!}\,
\frac{3\,(5532n^4+5464n^3+2173n^2+416n+32)}{(10n+1)(10n+3)}
=\frac{240}{\pi^2}.
\end{equation}
Guillera's bilateral sums \cite[\S5]{guillera2016bilateralsums},
\cite{guillera2020bilateral} relate such convergent series to divergent
ones, here to the divergent series for $1280/\pi^4$. We have not found
\eqref{eq:companion-intro} in the literature.
Specializations at roots of unity give Borwein--Bradley-type identities
(Section~\ref{sec:odd}), whose finite products depend only on $x^2$, $x^3$,
$x^4$ or $x^5$, as in Borwein and Bradley's generating function for
$\zeta(4 n + 3)$ \cite{borweinbradley1997,almkvistgranville1999}.
The quadratic and quartic ones generate $\zeta(2 n + 5)$ and $\zeta(4 n + 5)$, the cubic one generates $\zeta(6 n + 5)$ and $\pi^{6 n + 8}$, and the quintic one generates polynomials in $\zeta(10 n + 5)$ and $\pi^{10}$.

Comparing Taylor coefficients in these identities proves five of the six identities conjectured in \cite[Example X]{au2026harmonic}.
A typical example is
\begin{equation}\label{eq:zeta7}
\sum_{n \ge1}\frac{a_n}{n^9} \left[
p(n)(11 H_{2 n}^{(2)}-7 H_n^{(2)}) + \frac{p_7(n)}{4 n^2}
\right] = -3901440 \, \zeta(7),
\end{equation} 
where $H_n^{(r)} = \sum_{k = 1}^n k^{-r}$ and $p_7(n)=76632n^4 -87620n^3+37835 n^2 - 7225n + 510$ (see Corollary~\ref{cor:zeta7}).

The proof follows Guillera's periodicity argument
\cite{guillera2016bilateralsums,guillera2020bilateral,guillera2025flawless},
applied to both sums together, and does not assume \eqref{eq:zhao} or its companion.

\section{A parametric extension of Zhao's series}\label{sec:extension}

Let
\[
\HH = \{\theta\in\C^5 : \theta_1 + \cdots + \theta_5 = 0\},
\]
and let $e_j = e_j(\theta)$ be the elementary symmetric functions.
Let $P \in\Z[\nu,e_2,\dots,e_5]$ be the polynomial in \eqref{eq:P}, and define the summand
\begin{equation}\label{eq:S}
\calS(\nu;\theta) = 32 P(\nu;\theta)
\frac{\prod_i \Gamma(2 \nu - 1 + \theta_i) \Gamma(\nu - 2 \theta_i)}
{\Gamma(5 \nu + 1) \prod_{i<j}
	\Gamma(\nu + \tfrac12 - \theta_i - \theta_j)}.
\end{equation}
Note that this summand is symmetric in $\theta_1,\ldots, \theta_5$.

\begin{theorem}\label{thm:main}
For $\theta\in\HH$, as meromorphic identities,
\begin{align}
\sum_{n\geq 0}(-1)^n \calS(n;\theta)
& = 256\,\frac{\sum_{i = 1}^5 \cos2 \pi\theta_i - 1}{\prod_{i=1}^5 \sin\pi\theta_i},
\label{eq:first}\\
\sum_{n\geq 0}(-1)^n \calS(n + \tfrac12 + \theta_1 + \theta_2;\theta)
& = 2048\sec\bra{\pi(\theta_1 - \theta_2)}.\label{eq:pair}
\end{align}
Both series converge locally uniformly off the poles of their terms.
\end{theorem}

We prove Theorem~\ref{thm:main} in Section~\ref{sec:proof}. The summand originates from the WZ seed $\mathsf{Seed7}$, which is used in \cite[Example X]{au2025wzseeds} to prove \eqref{eq:zhao}, and in \cite[Example X]{au2026harmonic} to conjecture its harmonic-number twists. In the parameters $a,b,c,d,e$ of \cite{au2026harmonic}, put $\theta_1=a-b$, $\theta_2 = a + e$, $\theta_3=c+e$, $\theta_4=d+e$ and $\theta_5= 2a-b- c -d+e$.  Then $\HH$ corresponds to $b= 2a + 2e$. The symmetry of $\calS$ is not visible in the original coordinates.
In terms of $e_j = e_j(\theta)$, the polynomial $p$ of the introduction, and
\begin{equation}\label{eq:line-polynomials}
\begin{aligned}
b(n)& = 2868 n^4 - 3824 n^3 + 1881 n^2 - 395 n + 30,&
c(n)& = 483 n^2 - 245 n + 30, \\
u(n)& = 924 n^3 - 968 n^2 + 319 n - 29,&
w(n)& = 111 n^2 - 73 n + 10,\\ 
d(n)& = 5236 n^4 - 8376 n^3 + 4757 n^2 - 1119 n + 88, \\ 
q(n)& = 5148 n^5 - 9260 n^4 + 6485 n^3 - 2255 n^2 + 405 n - 32,
\end{aligned}
\end{equation}
the polynomial $P=P(n;\theta)$ is given by
\begin{align}
P = {}&n(2 n - 1)^5 p(n) + 8 e_2 n(2 n - 1)^3 b(n) + 16 e_2^2 n(2 n - 1)^3 c(n)\notag\\
& + 16 e_3 n^2(2 n - 1)^2 u(n)
+ 64 e_2 e_3 n^2(2 n - 1)^2(69 n - 11)
- 192 e_3^2 n(2 n - 1) w(n)\notag\\
& - 16 e_4 n(2 n - 1) d(n)
+ 64 e_2 e_4 n(2 n - 1)(59 n^2 - 37 n + 8)
+ 128 e_3 e_4 n(37 n^2 - 19 n + 4)\notag\\
& + 512 e_4^2 n(2 n - 1) - 32 e_5 q(n)
- 128 e_2 e_5(113 n^3 - 131 n^2 + 53 n - 8)\notag\\
& - 256 e_3 e_5(42 n^2 - 21 n + 4)
- 512 e_4 e_5(3 n - 2) - 1024 e_5^2.\label{eq:P}
\end{align}
 For $n \geq 1$, we also use the normalized terms
\begin{equation}\label{eq:T}
T(n;\theta) = \frac{a_n P(n;\theta)}{n^{10} (2 n - 1)^5}\;
\frac{\displaystyle\prod_{i = 1}^5
	\left[\prod_{k = 1}^{2 n - 2}\left(1 + \frac{\theta_i}{k}\right)
	\prod_{k = 1}^{n - 1}\left(1 - \frac{2 \theta_i}{k}\right) \right]}
{\displaystyle\prod_{i<j}\prod_{k = 1}^{n}
	\left(1 - \frac{\theta_i + \theta_j}{k - \frac12}\right)}.
\end{equation}
Since $P(n;0) = n(2n-1)^5 p(n)$, we have $T(n;0) = a_n\,p(n)/n^9$. 
\begin{corollary}\label{cor:main}
On $\HH$, one has
\begin{equation}\label{eq:main}
1 + \frac{e_5(\theta)}{1024}\sum_{n \ge1}T(n;\theta)
= \frac{\sum_{i = 1}^5 \cos2 \pi\theta_i - 1}{4}\,
\frac{\prod_{i<j}\Gamma(\tfrac12 - \theta_i - \theta_j)}
{\pi^{5/2}\prod_i \Gamma(\tfrac12 - \theta_i)}.
\end{equation}
The series converges absolutely and locally uniformly away from
$\theta_i + \theta_j \in\tfrac12 + \Z_{\ge0}$.
\end{corollary}

\begin{proof} Write
\begin{equation}
K(\theta) = \frac{\sum_{i = 1}^5 \cos2 \pi\theta_i - 1}{4},\qquad
G(\theta) = \frac{\prod_{i<j}\Gamma(\tfrac12 - \theta_i - \theta_j)}
{\pi^{5/2}\prod_i \Gamma(\tfrac12 - \theta_i)}.
\end{equation}
By Gamma reflection and duplication, for $\theta\in\HH$ and $n\geq 1$,
\begin{equation}\label{eq:normalization}
e_5\, T(n;\theta) = (-1)^n\calS(n;\theta)\, G(\theta) \prod_i \sin\pi\theta_i,
\end{equation}
and for $n = 0$ the right side equals $1024$, since
$P(0;\theta) = 1024 e_5 \prod_i(1 - \theta_i)$. Dividing by $1024$,
summing over $n\geq 0$ and using \eqref{eq:first} gives \eqref{eq:main}.
\end{proof}

Comparing coefficients of $e_5$ in \eqref{eq:main} gives Zhao's
series \eqref{eq:zhao}. It is also given as the constant term of each identity in Corollary~\ref{cor:lines}.

\begin{corollary}\label{cor:companion} We have
\begin{equation}\label{eq:companion}
\sum_{n \geq 0}\widehat a_n\,p(n + \tfrac12) = \frac{240}{\pi^2},
\qquad
\widehat a_n = \left(-\frac{1024}{3125}\right)^n
\frac{(\tfrac12)_n^9}{(1)_n^5(\tfrac7{10})_n(\tfrac9{10})_n
	(\tfrac{11}{10})_n(\tfrac{13}{10})_n},
\end{equation}
where $p(n + \tfrac12) = 5532n^4 +5464n^3 +2173n^2 + 416n + 32$.
\end{corollary}

\begin{proof}
At $\theta = 0$, Gamma duplication and multiplication give
\[
(-1)^n \calS(n + \tfrac12;0)
= \frac{128 \pi^2}{15}\widehat a_n\,p(n + \tfrac12).
\]
For $n = 0$ this holds as a limit $\theta\to0$: with $s = \theta_1 + \theta_2$,
the factor $\prod_i(2 s + \theta_i)$ of $P(\tfrac12 + s;\theta)$ cancels the
poles of the $\Gamma(2 s + \theta_i)$. The other terms are regular at $\theta = 0$. Summing and applying \eqref{eq:pair} at the
origin gives \eqref{eq:companion}.
\end{proof}

Identity \eqref{eq:companion} is equivalent to \eqref{eq:companion-intro}.
Thus \eqref{eq:zhao} and \eqref{eq:companion} come from the same summand $\calS$, summed at integers and at half-integers.
As for other Ramanujan-type series \cite{zudilin2009,cohenguillera2021},
\eqref{eq:companion} appears to have a Zudilin-type $p$-adic analogue:
\begin{equation}
\sum_{k=0}^{p-1}\widehat a_k\,p(k+\tfrac12)\stackrel{?}{\equiv} 32p^2 \pmod{p^5},
\end{equation}
for primes $p>7$. 

\section{Borwein--Bradley-type identities}\label{sec:odd}

Many Ap\'ery-like series for zeta values are Taylor coefficients of
identities whose summands contain finite products in one or more
parameters. The first is the identity found independently by Koecher
\cite{koecher80} and Leshchiner \cite{leshchiner81}, which generates the values
$\zeta(2n+3)$ and has Ap\'ery's series for $\zeta(3)$ as its constant term.
Borwein and Bradley \cite{borweinbradley1997} found a generating function for
$\zeta(4n+3)$ whose products depend only on $x^4$, later proved by Almkvist and
Granville \cite{almkvistgranville1999}. Bradley \cite{bradley2002} and Rivoal
\cite{rivoal2004simultaneous} showed that both identities are specializations
of a two-parameter identity conjectured by H.~Cohen. Further identities of
this kind were found by Bailey, Borwein and Bradley \cite{bbb06} and by the
Hessami Pilehroods \cite{pp08,pp11}, including a two-parameter extension of
the Amdeberhan--Zeilberger series \cite{pp11}. We call identities whose
products depend only on one power $x^r$ \emph{Borwein--Bradley-type}. In
\cite{shalev2026bb}, the series of Cullen and Zhao for $1/\pi^4$ are constant
terms of such identities with $r = 3, 4, 6, 8$. Here, Zhao's series
\eqref{eq:zhao} for $\zeta(5)$ is the constant term of such identities with
$r = 2, 3, 4, 5$.

We now restrict \eqref{eq:main} to lines on which its finite
products are products in one power of $x$.
Such lines come from roots of unity: if
\[
\theta = x \cdot (1,\omega_r,\ldots,\omega_r^{r - 1},0,\ldots,0),
\qquad \omega_r = e^{2 \pi i/r},\qquad 2 \leq r \leq 5,
\]
then replacing $x$ by $\omega_r x$ permutes the coordinates. Therefore, every product over coordinates or pairs, and the polynomial $P$, depends only on $x^r$. For $r \leq 4$ the line lies on
$e_5 = 0$, so we recover the sum by perturbing the line and comparing
coefficients in \eqref{eq:main}.

\begin{corollary}\label{cor:lines}
On the quadratic line $\theta = (x,-x,0,0,0)$, for $|x|<\frac12$,
\begin{multline}\label{eq:odd-generating}
\sum_{n \geq 1}\frac{a_n}{n^9}
\left[p(n) - \frac{8 x^2 \bigl(b(n) - 2 x^2 c(n) \bigr)}{(2 n - 1)^2}\right]
\frac{\displaystyle\prod_{k = 1}^{2 n - 2}\bra{1 - x^2/k^2}
	\prod_{k = 1}^{n - 1}\bra{1 - 4 x^2/k^2}}
{\displaystyle\prod_{k = 1}^{n}\bra{1 - x^2/(k - \tfrac12)^2}^3}
\\
= -1024\sum_{m \ge0}(2 m + 3)(2 m + 4)(2^{2 m + 5} - 1)\z(2 m + 5) x^{2 m}.
\end{multline}
On the cubic line $\theta = x(1,\omega_3,\omega_3^2,0,0)$, for $|x|<\frac12$,
\begin{multline}\label{eq:cubic-generating}
\sum_{n \geq 1}\frac{a_n}{n^9}
\left[p(n) + \frac{16 x^3 n\,u(n)}{(2 n - 1)^3} - \frac{192 x^6 w(n)}{(2 n - 1)^4}\right]
\frac{\displaystyle\prod_{k = 1}^{2 n - 2}\bra{1 + x^3/k^3}
	\prod_{k = 1}^{n - 1}\bra{1 - 8 x^3/k^3}}
{\displaystyle\prod_{k = 1}^{n}\bra{1 + x^3/(k - \tfrac12)^3}\bra{1 - x^3/(k - \tfrac12)^3}^2}
\\
=-6144\sum_{m \ge0}(3 m + 2)(2^{6 m + 5} - 1) \z(6 m + 5) x^{6 m}
\\
+ \frac{512 \pi^2}{x^3}\left[1 + 2 \prod_{j = 0}^2 \sec(\pi\omega_3^j x)
- \sum_{j = 0}^2 \sec^2(\pi\omega_3^j x)\right].
\end{multline}
On the quartic line $\theta = (x,-x,ix,-ix,0)$, for
$|x|<\frac1{2 \sqrt2}$,
\begin{multline}\label{eq:quartic-generating}
\sum_{n \geq 1}\frac{a_n}{n^9}
\left[p(n) + \frac{16 x^4 \bigl(d(n) + 32 x^4 \bigr)}{(2 n - 1)^4}\right]
\frac{\displaystyle\prod_{k = 1}^{2 n - 2}\bra{1 - x^4/k^4}
	\prod_{k = 1}^{n - 1}\bra{1 - 16 x^4/k^4}}
{\displaystyle\prod_{k = 1}^{n}\bra{1 - x^4/(k - \tfrac12)^4}\bra{1 + 4 x^4/(k - \tfrac12)^4}}
\\
= 2048 \sum_{m \geq 0}\bigl((-4)^{m + 1} - 2 \bigr)(2^{4 m + 5} - 1)
\zeta(4 m + 5) x^{4 m}.
\end{multline}
On the quintic line $\theta = x\cdot (1,\omega,\omega^2,\omega^3,\omega^4)$,
where $\omega = \omega_5 = e^{2 \pi i/5}$,
\begin{multline}\label{eq:quintic-generating}
\sum_{n \geq 1}\frac{a_n}{n^9}
\left[p(n) - \frac{32 x^5 \bigl(q(n) + 32 x^5 \bigr)}{n(2 n - 1)^5}\right]
\frac{\displaystyle\prod_{k = 1}^{2 n - 2}\bra{1 + x^5/k^5}
	\prod_{k = 1}^{n - 1}\bra{1 - 32 x^5/k^5}}
{\displaystyle\prod_{k = 1}^{n}\bra{1 + 11 x^5/(k - \tfrac12)^5 - x^{10}/(k - \tfrac12)^{10}}}
\\
= 256\pi^5\,
\frac{\sum_{j = 0}^4 \cos(2 \pi\omega^j x) - 1}
{\prod_{j = 0}^4 \sin(\pi\omega^j x)}
\prod_{k \geq 1}
\frac{\bra{1 + x^5/k^5}\bra{1 - 32 x^5/k^5}}{1 + 11 x^5/(k - \tfrac12)^5 - x^{10}/(k - \tfrac12)^{10}} - \frac{1024}{x^5}.
\end{multline}
The quintic identity holds meromorphically, with the value at $x = 0$ understood by continuity.
\end{corollary}

\begin{proof}
For \eqref{eq:odd-generating}, put
$\theta = (x,-x,\varepsilon,\varepsilon,-2 \varepsilon)$
in \eqref{eq:main} and compare coefficients of $\varepsilon^3$.
For \eqref{eq:cubic-generating}, use
$\theta = (x,\omega_3 x,\omega_3^2 x,\varepsilon,-\varepsilon)$
and compare coefficients of $\varepsilon^2$.
For \eqref{eq:quartic-generating}, use
$\theta = (x - \varepsilon/4,-x - \varepsilon/4,
ix - \varepsilon/4,-ix - \varepsilon/4,\varepsilon)$
and compare coefficients of $\varepsilon$.
The displayed zeta series follow from the Taylor expansions of the
polygamma functions $\psi_k = (d/dz)^{k + 1}\log\Gamma(z)$; for the cubic
line, $\psi_1(\tfrac12 - z) + \psi_1(\tfrac12 + z) = \pi^2\sec^2 \pi z$ gives
the trigonometric part.
For \eqref{eq:quintic-generating}, substitute the quintic line
directly: $e_2 = e_3 = e_4 = 0$, $e_5 = x^5$, and
\[
\prod_{i<j}\bra{1 - \frac{\theta_i + \theta_j}{k - \tfrac12}}
= 1 + \frac{11 x^5}{(k - \tfrac12)^5} - \frac{x^{10}}{(k - \tfrac12)^{10}}.
\]
Euler's Gamma product, duplication and reflection give the
right-hand side.
\end{proof}

The right-hand sides of \eqref{eq:odd-generating} and
\eqref{eq:cubic-generating} also have the forms (respectively)
\[
\frac{512}{x^2}\sqbra{\psi_2\bra{\tfrac12 + x} + \psi_2\bra{\tfrac12 - x} + 28 \zeta(3) },
\qquad
\frac{1024}{x^3}\sqbra{\frac{\pi^2}{2}
+ \pi^2 \prod_j \sec(\pi\omega_3^j x)
- \sum_{j = 0}^2 \psi_1\bra{\tfrac12 - \omega_3^j x} },
\]
which continue them meromorphically. In the cubic identity, the coefficient of $x^{6 n + 3}$ is a rational multiple of $\pi^{6 n + 8}$. The constant terms of the four identities give \eqref{eq:zhao}.

The coefficients of the quintic identity can be written explicitly as well. Let
$\phi = (1 + \sqrt5)/2$, let $L_n = \phi^n + (-\phi)^{-n}$ be the Lucas numbers, and put
\[
c_m = (-1)^{m + 1} - 32^m + (-1)^m(2^{5 m} - 1) L_{5 m}.
\]
Since $1+11y-y^2 = \bra{1 - \phi^{-5}y}\bra{1 + \phi^5 y}$ and
$\prod_j \sin(\pi\omega^jx) = \pi^5 x^5 \prod_{k \ge1}\bra{1 - x^{10}/k^{10}}$,
the sum in
\eqref{eq:quintic-generating} equals
\begin{equation}\label{eq:quintic-coefficient}
\begin{split}
&\frac{1024}{x^5}\left[
\frac{\sum_{j = 0}^4 \cos(2 \pi\omega^jx) - 1}{4 \prod_{k \ge1}\bra{1 - x^{10}/k^{10}}}
\exp\Bigl(\sum_{m \ge1}\frac{c_m \zeta(5 m)}{m}x^{5 m}\Bigr) - 1 \right]\\
&\quad = -380928 \zeta(5)
+ \left(70852608 \zeta(5)^2 + \frac{2048}{3}\pi^{10}\right) x^5\\
&\qquad - \left(15266800640 \zeta(15) + 8785723392 \zeta(5)^3
+ 253952 \pi^{10}\zeta(5) \right) x^{10} + O(x^{15}).
\end{split}
\end{equation}
Thus the coefficient of $x^{5 N}$ is a polynomial in
$\zeta(10 m + 5)$ and $\pi^{10}$, homogeneous of
weight $5 N + 5$.

Comparing Taylor coefficients proves five of the six identities
conjectured in \cite[Example X]{au2026harmonic}. Put
\[
\begin{aligned}
p_8(n)& = 238512 n^5 - 351796 n^4 + 213070 n^3 - 65625 n^2 + 10185 n - 630,\\
p_9(n)& = 6044128 n^6 - 11644928 n^5 + 9540232 n^4 - 4214800 n^3
+ 1052415 n^2 - 140065 n + 7710,\\
p_{10}(n)& = 43911744 n^7 - 112688016 n^6 + 121035008 n^5 - 71405368 n^4\\
&\qquad + 25155570 n^3 - 5304255 n^2 + 618915 n - 30690.
\end{aligned}
\]

\begin{corollary}\label{cor:zeta7}
Identity \eqref{eq:zeta7} holds, and
\begin{align}
\sum_{n\geq 1}\frac{a_n}{n^9}\left[
p(n)(H_{2 n}^{(3)} - H_n^{(3)}) + \frac{p_8(n)}{24 n^3(2 n - 1)}
\right]& = -\frac{512}{9}\pi^8,\label{eq:pi8}\\
\sum_{n\geq 1}\frac{a_n}{n^9}\left[
p(n)(13 H_n^{(4)} + 49 H_{2 n}^{(4)}) - \frac{p_9(n)}{16 n^4(2 n - 1)^2}
\right]& = -14651392\,\zeta(9),\label{eq:zeta9}\\
\sum_{n\geq 1}\frac{a_n}{n^9}\left[
p(n)(7 H_n^{(5)} + 117 H_{2 n}^{(5)}) - \frac{p_{10}(n)}{96 n^5(2 n - 1)^3}
\right]& = -23617536\,\zeta(5)^2 - \frac{2048}{9}\pi^{10}.\label{eq:zeta5sq}
	\end{align}
\end{corollary}

\begin{proof}
Compare coefficients of $x^2$ in \eqref{eq:odd-generating}, of $x^3$ in \eqref{eq:cubic-generating}, of $x^4$ in \eqref{eq:quartic-generating} and of $x^5$ in \eqref{eq:quintic-generating}, using \eqref{eq:quintic-coefficient}   for the last.
\end{proof}

The fifth identity of \cite[Example X]{au2026harmonic}, with value
$-\frac{44552192}{5}\zeta(9)$ and a quadratic expression in
$H_n^{(2)}$ and $H_{2 n}^{(2)}$, is $\frac25$ times the
coefficient of $x^4$ in \eqref{eq:odd-generating} plus
$\frac{47}{245}$ times that in \eqref{eq:quartic-generating}.
The sixth involves first-order harmonic numbers and
$\zeta(\overline5,1)$, and presumably requires the full
five-parameter family of \cite{au2026harmonic}, off the hyperplane $\HH$.
Higher coefficients give further identities with products of
harmonic numbers.

\section{Proof of Theorem~\ref{thm:main}}\label{sec:proof}

We follow Guillera's periodicity argument, as in \cite[Section~7]{shalev2026bb}. Let
\begin{equation}\label{eq:Q}
Q(\theta) = \frac{\prod_i \sin\pi\theta_i}{1024}
\sum_{n \geq 0}(-1)^n \calS(n; \theta),
\end{equation}
so that \eqref{eq:first} states $Q = K$.
Define $r_{ij} \coloneqq \delta_i - \delta_j$, where $\delta_i$ is a coordinate vector. For the shift $r_{12}$, we have
\begin{equation}\label{eq:certificate}
\calS(\nu;\theta + r_{12}) - \calS(\nu;\theta) = - \rho_{12}(\nu + 1;\theta) \calS(\nu + 1;\theta) - \rho_{12}(\nu;\theta) \calS(\nu;\theta),
\end{equation}
where the certificate is
\begin{equation}\label{eq:b}
\rho_{12}(\nu;\theta) =
\frac{640(\theta_2 - \theta_1 - 1) \nu\prod_{\ell = 1}^4(5 \nu - \ell)
	\prod_{i<j}(\nu - \tfrac12 - \theta_i - \theta_j)}
{(2 \nu - 2 + \theta_2)(\nu - 1 - 2 \theta_1)(\nu - 2 - 2 \theta_1) P(\nu;\theta)
	\prod_{j = 3}^5(\nu - \tfrac12 - \theta_2 - \theta_j)},
\end{equation}
which is symmetric in $\theta_3,\theta_4,\theta_5$. Dividing \eqref{eq:certificate} by $\calS(\nu;\theta)$ turns it into an identity of rational functions of $\nu$ and $\theta_1,\ldots,\theta_4$, which is trivially checked. Since $\calS$ is symmetric, the analogous certificate for $r_{ij}$ is $\rho_{ij}(\nu;\theta) = \rho_{12}(\nu;\theta_{\sigma(1)},\ldots,\theta_{\sigma(5)})$
for any permutation $\sigma$ with $\sigma(1) = i$ and $\sigma(2) = j$.
Summing \eqref{eq:certificate} with alternating signs from a starting
point $\alpha$, we obtain
\begin{equation}\label{eq:tel}
\sum_{n \ge0}(-1)^n
\bigl[\calS(\alpha + n;\theta + r_{ij}) - \calS(\alpha + n;\theta) \bigr]
= -\rho_{ij}(\alpha;\theta) \calS(\alpha;\theta),
\end{equation}
where the boundary term vanishes at infinity because the terms decay geometrically.
 By the factors $\nu$ and $\nu - \tfrac12 - \theta_1 - \theta_2$ of
\eqref{eq:b} and their permutations, we have $\rho_{ij}(0;\theta)=0$ for $i\neq j$, and 
\begin{equation}\label{eq:zeros}
\rho_{12}(\tfrac12 + \theta_1 + \theta_2;\theta) = \rho_{34}(\tfrac12 + \theta_1 + \theta_2;\theta) = \rho_{13}(\tfrac12 + \theta_1 + \theta_2;\theta) = 0.
\end{equation}

\begin{lemma}\label{lem:ambiguity}
The function $Q$ is entire and invariant under the translations
$\theta\mapsto\theta + r_{ij}$, and
$Q = K$ wherever some $\theta_i$ is an integer.
\end{lemma}

\begin{proof}
By \eqref{eq:tel} with $\alpha = 0$, the sum in \eqref{eq:Q} is
invariant under $\theta\mapsto\theta + r_{ij}$. So is the product $\prod_i \sin\pi\theta_i$, since two of its factors change sign. Hence $Q$ is periodic. The terms of the sum have possible poles only on
$\theta_i \in\Z$  and on $\theta_i \in\tfrac12 + \Z_{\ge0}$, and by Hartogs' theorem it suffices to check generic points of these hyperplanes. The former are cancelled by the zeros of $\sin\pi\theta_i$ and, for $n = 0$ at $\theta_i \in\{0,1\}$, by the factor $e_5 \prod_i(1 - \theta_i)$ of $P(0;\theta)$. The latter translate by periodicity to the hyperplane $\theta_i = -\tfrac12$, where the terms are regular.
Hence $Q$ is entire.

On $\theta_1 = 0$ only the term $n = 0$ contributes to $Q$, as the terms with $n \ge1$ are regular there and are multiplied by $\sin\pi\theta_1 = 0$. By
\eqref{eq:normalization}, the term $n = 0$ contributes $1/G$, so using $\theta_2 + \theta_3 + \theta_4 + \theta_5 = 0$, 
reflection and the product-to-sum formula give
\[
Q = \frac1 G = \prod_{j = 3}^5 \cos\pi(\theta_2 + \theta_j)
= \frac14 \sum_{j = 2}^5 \cos2 \pi\theta_j = K.
\]
By symmetry and periodicity, $Q - K$ vanishes whenever some
$\theta_i$ is an integer.
\end{proof}

Define
\[
U(\theta) = \sum_{n \ge 0}(-1)^n \calS(n + \tfrac12 + \theta_1 + \theta_2;\theta).
\]

\begin{lemma}\label{lem:growth}
For fixed generic $\theta\in\HH$, $U(\theta + zr_{13})\to0$ as
$|\operatorname{Im}z|\to\infty$, uniformly for $0 \le\operatorname{Re}z \le1$.
\end{lemma}

\begin{proof}
Throughout, $0 \le\operatorname{Re}z \le1$, $Y = |\operatorname{Im}z|\ge Y_0$, and
$C$ and $D$ denote constants, which may depend on $\theta$ and may change
from line to line.
We need a bound on the terms of $U(\theta + zr_{13})$ that is uniform in
both $m$ and $Y$, and obtain it from Stirling's formula. The $m$-th term is
$\calS(m + \tfrac12 + \theta_1 + \theta_2 + z;\theta + zr_{13})$, and each
Gamma factor in \eqref{eq:S} then has argument $\lambda m + i \mu Y + \gamma$,
with $\lambda \ge 1$, $\mu\in\Z$ and $\gamma$ bounded. For such a factor,
uniformly in fixed sectors,
\[
\log\bigl|\Gamma(\lambda m + i \mu Y + \gamma) \bigr|
= Yf_{\lambda,|\mu|}(m/Y) + \lambda m \log Y - \lambda m + O \bigl(\log(2 + m + Y) \bigr),
\]
where $f_{\lambda,\mu}(t) = \lambda t \log|\lambda t + i \mu| - \mu\arg(\lambda t + i \mu)$
for $t \ge 0$, with continuous values at $t = 0$. For the ten factors in the
numerator of \eqref{eq:S} and the eleven in the denominator, the pairs
$(\lambda,|\mu|)$ are
\[
\begin{array}{ll}
\text{numerator:}&(2,1),\ (2,2)^3,\ (2,3),\
(1,1)^4,\ (1,3),\\
\text{denominator:}&(5,5),\ (1,0)^3,\ (1,1)^4,\ (1,2)^3,
\end{array}
\]
where superscripts denote multiplicities. Summing over the factors, we get
$\log|\calS| = Y g(m/Y) + O(\log(2 + m + Y))$ with
$g = \sum_{\rm num}f_{\lambda,|\mu|} - \sum_{\rm den}f_{\lambda,|\mu|}$, since the
coefficients of $\nu$ in the Gamma arguments of \eqref{eq:S} add up to $15$ in
both the numerator and the denominator. Moreover, the factors with $\mu\ne0$ stay in fixed Stirling sectors, and
those with $\mu = 0$ occur only in the denominator. Hence, uniformly for
$m \ge0$,
\begin{equation}\label{eq:profile}
|\calS(m + \tfrac12 + \theta_1 + \theta_2 + z;\theta + zr_{13})|
\le C(2 + m + Y)^D \exp\{Yg(m/Y)\}.
\end{equation}

It remains to bound $g$. Direct differentiation shows that $g(0) = -\pi$,
that $g \le 3 - \pi$ on $[0,1]$, and that $g$ decreases at least linearly
on $[1,\infty)$. Summing \eqref{eq:profile} over $m$ and splitting at
$m = Y$, we obtain
\[
U(\theta + zr_{13}) = O \bigl((1 + Y)^{D + 1}e^{(3 - \pi) Y}\bigr),
\]
which tends to $0$ since $\pi > 3$.
\end{proof}

\begin{lemma}\label{lem:pair}
	Identity \eqref{eq:pair} holds.
\end{lemma}

\begin{proof}
Notice that the shifts $r_{12}$ and $r_{34}$ fix $\theta_1+\theta_2$. Hence, by \eqref{eq:tel} and \eqref{eq:zeros}, $U$ is invariant under $r_{12}$ and $r_{34}$. The shift $r_{13}$ raises $\theta_1 + \theta_2$ by one, and \eqref{eq:tel} gives $U(\theta + r_{13}) = -U(\theta)$. Since $U$ is symmetric in $\theta_1,\theta_2$ and in
$\theta_3,\theta_4,\theta_5$, it follows that

\begin{equation}\label{eq:pairshifts}
U(\theta + r_{ij}) = \pm U(\theta),
\end{equation}
with the minus sign when exactly one of $i,j$ lies in $\{1,2\}$.
The terms of $U$ have possible poles on
$2 \theta_1 + 2 \theta_2 + \theta_i \in\Z_{\le0}$ ($i = 1,2$), on
$2 \theta_1 + 2 \theta_2 + \theta_j = 0$ and
$\theta_1 + \theta_2 - 2 \theta_j \in - \tfrac12 - \Z_{\ge0}$ ($j = 3,4,5$),
and on $\theta_1 - \theta_2 \in\tfrac12 + \Z$. By \eqref{eq:pairshifts}, each of the first three families translates to a hyperplane on which the terms are regular, so these poles are removable. The remaining poles are simple, and by \eqref{eq:pairshifts} it suffices to consider
 $\theta_1 - \theta_2 = \tfrac12$, where only the term $n = 0$ is singular.
Putting $\theta_2 = \theta_1 - \tfrac12 - \delta$, Gamma recurrence and duplication give
\begin{equation}\label{eq:residue}
\lim_{\delta\to0}\delta U = -\frac{256 P(2 \theta_1;\theta)}{\pi R} = -\frac{2048}{\pi},
\qquad
R = \prod_{\ell = 0}^3(10 \theta_1 - \ell) \prod_{j = 3}^5(4 \theta_1 - 1 + \theta_j)(2 \theta_1 - 2 \theta_j),
\end{equation}
since $P(2 \theta_1;\theta) = 8 R$ when $\theta_1 - \theta_2 = \tfrac12$.
This is the residue of $2048\sec\bra{\pi(\theta_1 - \theta_2)}$, which
satisfies the same shift law \eqref{eq:pairshifts}. So the difference 
$U(\theta) - 2048\sec\bra{\pi(\theta_1 - \theta_2)}$ is entire. By Lemma~\ref{lem:growth}, $U(\theta + zr_{13})$ tends to $0$
as $|\operatorname{Im}z|\to\infty$,  uniformly for $0 \le\operatorname{Re}z \le1$, while $\theta\mapsto\sec\bra{\pi (\theta_1-\theta_2)}$ is bounded there. Hence the entire function $U(\theta + zr_{13}) - 2048\sec\bra{\pi(\theta_1 - \theta_2 + z)}$ changes sign under $z \mapsto z + 1$, by \eqref{eq:pairshifts}, and is $o(e^{\pi|\operatorname{Im}z|})$ on every vertical strip. Hence its Fourier series contains only the terms $e^{i\pi kz}$ with $k$ odd,
and each of these grows at least like $e^{\pi|\operatorname{Im}z|}$ at one end
of the strip, so the growth bound forces all coefficients to vanish. Therefore $U = 2048\sec\bra{\pi(\theta_1 - \theta_2)}$, as wished.
\end{proof}

\begin{proof}[Proof of Theorem~\ref{thm:main}]
Identity \eqref{eq:pair} is Lemma~\ref{lem:pair}. For \eqref{eq:first},
note that on the hyperplane $\theta_1 + \theta_2 = -\tfrac12$ the sum
$U(\theta)$ coincides with the sum in \eqref{eq:Q}. Hence, using
$\theta_3 + \theta_4 + \theta_5 = \tfrac12$, Lemma~\ref{lem:pair} and the
product-to-sum formula give
\[
Q = 2\sec\bra{\pi(\theta_1 - \theta_2)} \prod_i \sin\pi\theta_i
= \prod_{j = 3}^5 \sin\pi\theta_j = K.
\]
By periodicity, $Q = K$ also on $\theta_1 + \theta_2 \in -\tfrac12 + \Z$.

Now fix a generic $\theta\in\HH$. By Lemma~\ref{lem:ambiguity}, the
function $z \mapsto (Q - K)(\theta + zr_{13})$ is entire and $1$-periodic.
Along $\theta + zr_{13}$, with $Y = |\operatorname{Im}z|$, the ratio of
consecutive Gamma quotients in \eqref{eq:S} is at most
\[
\frac{64}{3125}\,\frac{(4 n^2 + Y^2)^2(n^2 + 4 Y^2)}{(n^2 + Y^2)^3}\Bigl(1 + O\bigl(\tfrac1 n\bigr)\Bigr)
\le\frac{1216}{3125}\Bigl(1 + O\bigl(\tfrac1 n\bigr)\Bigr),
\]
uniformly in $Y$, and Stirling's formula bounds the first terms
polynomially in $Y$. Since $P$ is a polynomial, it follows that, for some
$D \ge0$,
\begin{equation}\label{eq:Qgrowth}
(Q - K)(\theta + zr_{13}) = O \bigl((1 + Y)^D e^{2 \pi Y}\bigr)
\end{equation}
for $0 \le\operatorname{Re}z \le1$. Hence its Fourier series contains
only the terms $e^{2 \pi ikz}$ with $|k| \le 1$, since the others grow
faster than $e^{2 \pi Y}$ at one end of the strip. Thus
$e^{2 \pi iz}(Q - K)(\theta + zr_{13})$ is a polynomial of degree at most
$2$ in $e^{2 \pi iz}$. On the other hand, by Lemma~\ref{lem:ambiguity} and
the first part of the proof, it vanishes where $\theta_1 + z\in\Z$, where
$\theta_3 - z\in\Z$, and where $\theta_1 + \theta_2 + z\in -\tfrac12 + \Z$,
that is, at three points that are distinct modulo $1$. Therefore
$(Q - K)(\theta + zr_{13}) = 0$, and $Q = K$ on $\HH$, as wished.
\end{proof}

\section*{Acknowledgements}
The author is grateful to Kam Cheong Au for his helpful comments and advice.

\section*{Use of AI tools}
The questions studied here and the restriction of the WZ seed to the
hyperplane on which its summand becomes symmetric are the author's. Large language models (Claude
Opus~5.5 by Anthropic and GPT-6 by OpenAI) were used, under the author's
direction, for experiments with WZ seeds and to work out the proofs. The
author checked all arguments, verified the certificate identity
\eqref{eq:certificate} and the coefficient comparisons of
Section~\ref{sec:odd} by exact computation, and takes full responsibility
for the paper.

{\footnotesize
	\bibliographystyle{unsrt}
	\bibliography{paper_refs}
}

\Address
\end{document}